\documentclass{amsart}
\usepackage{amscd}
\usepackage{amsfonts}
\usepackage{amsmath}
\usepackage{amssymb}
\usepackage{latexsym}
\usepackage{amsthm,color}
\usepackage[all]{xy}
\usepackage{epsfig}
\usepackage{graphicx}
\usepackage{color}
\usepackage{tikz}
\usetikzlibrary{cd}
\usetikzlibrary{calc}
\newtheorem{theorem}{Theorem}

\newtheorem{fact}{Fact}

\theoremstyle{definition}
\newtheorem{definition}{Definition}
\theoremstyle{definition}
\newtheorem{example}{Example}
\theoremstyle{definition}

\newtheorem{note}{Note}

\newtheorem{problem}{Problem}

\begin{document}
\title[
Envelopes created by circle families in the plane]
{Envelopes created by circle families in the plane}
\author[Y.~Wang]{Yongqiao Wang
}
\address{
School of Science, Dalian Maritime University, Dalian 116026, P.R. China
}
\email{wangyq@dlmu.edu.cn}
\author[T.~Nishimura]{Takashi Nishimura
}
\address{
Research Institute of Environment and Information Sciences,
Yokohama National University,
Yokohama 240-8501, Japan}
\email{nishimura-takashi-yx@ynu.ac.jp}
\begin{abstract}
In this paper,  on envelopes created by circle families in the plane,
all four basic problems (existence problem, representation problem,
problem on the number of envelopes, problem on relationships of definitions)
are solved.    
\end{abstract}


\date{}

\maketitle

\section{Introduction\label{section1}}
Throughout this paper, $I$ is an open interval and all functions,
mappings are of class $C^\infty$ unless otherwise stated.
\medskip
\par
Envelopes of planar regular curve families have fascinated many pioneers
since the dawn of differential analysis (for instance, see \cite{history}).
In most typical cases, straight line families have been studied.     
However, even for envelopes created by straight line falimies, 
there were several unsolved problems  
until very recently.    For instance, as the following 
Example \ref{example0} shows, 
the well-known method to represent the envelope 
for a given straight line family creating an envelope is useless in some cases.   
and the representation problem was open until very recently.    
\begin{example}\label{example0} 
Consider the elementary plane curve $f: \mathbb{R}\to \mathbb{R}^2$ 
defined by $f(t)=\left(t, t^3\right)$.     The regular curve $f$ 
gives a parametrization of the non-singular cubic curve 
\[
C=\left\{(X, Y)\in \mathbb{R}^2\, |\, Y=X^3\right\}.   
\]   
The affine tangent line $L_t$ to $C$ at a point $\left(t, t^3\right)$ 
may be defined by 
\[
\left(X-t, Y-t^3\right)\cdot \left(-3t^2, 1\right)=0,
\] 
where the dot in the center stands for the standard scalar product of 
two vectors $\left(X-t, Y-t^3\right)$ and 
$\left(-3t^2, 1\right)$.  
Since the straight line family $\left\{L_t\right\}_{t\in \mathbb{R}}$ 
is 
the affine tangent line family to $C$, the non-singular cubic curve $C$ must be 
an envelope of $\left\{L_t\right\}_{t\in \mathbb{R}}$. 
Set 
\[
F\left(X, Y, t\right)=\left(X-t, Y-t^3\right)\cdot \left(-3t^2, 1\right)=
-3t^2X+Y+2t^3.   
\]
We have the following.   
\begin{eqnarray*}
\mathcal{D} & = & 
\left\{(X, Y)\in \mathbb{R}^2\: \left|\: \exists t\in \mathbb{R} \mbox{ s.t. }
F(X, Y, t)=\frac{\partial F}{\partial t}(X, Y, t)=0\right.\right\} \\ 
{ } & = & 
\left\{(X, Y)\in \mathbb{R}^2\: \left|\: \exists t\in \mathbb{R} \mbox{ s.t. }
-3t^2X+Y+2t^3=-6t(X-t)=0\right.\right\} \\ 
{ } & = & \left\{(X, Y)\in \mathbb{R}^2\: 
\left|\: Y=0 \mbox{ or }Y=X^3\right.\right\} \\ 
{ } & \supsetneqq & \mathcal{C}.   
\end{eqnarray*}   
Therefore, 
unfortunately, 
the well-known method to represent the envelope does not work well 
in this case and it must be applied under appropriate assumptions.   
\end{example}
In \cite{nishimura}, by solving four basic problems on envelopes created by
straight line families in the plane (existence problem, representation problem,
uniqueness problem and equivalence problem of definitions),
the second author constructs a general theory for envelopes created by
straight line families in the plane.
On the other hand, 
circle families in the plane are non-negligible families  
because the envelopes of them have already had an important application,
namely, an application to
Seismic Survey.
Following 7.14(9) of \cite{brucegiblin}, a brief explanation of Seismic Survey
is given as follows.
In the Euclidean plane $\mathbb{R}^2$,
consider the \lq\lq ground level curve\rq\rq\,
$C$ parametrized by $\gamma: I\to \mathbb{R}^2$.
Suppose that there is a stratum of granite below the top layer of sandstone
and that the dividing curve, denoted by $M$,
is parametrized by $\widetilde{f}: I\to \mathbb{R}^2$.
Seismic Survey is the following
method to obtain an approximation of $\widetilde{f}$
as precisely as possible.
Take one fixed point $A$ of $C$ and consider an explosion at $A$.
Assume that the sound waves travel in straight lines and are reflected from $M$,
arriving back at points $\gamma(t)$ of $C$ where their times of arrival
are exactly recorded by sensors located along $C$ (see Figure 1).
\begin{figure}[h]
\begin{center}
\includegraphics[width=10cm]
{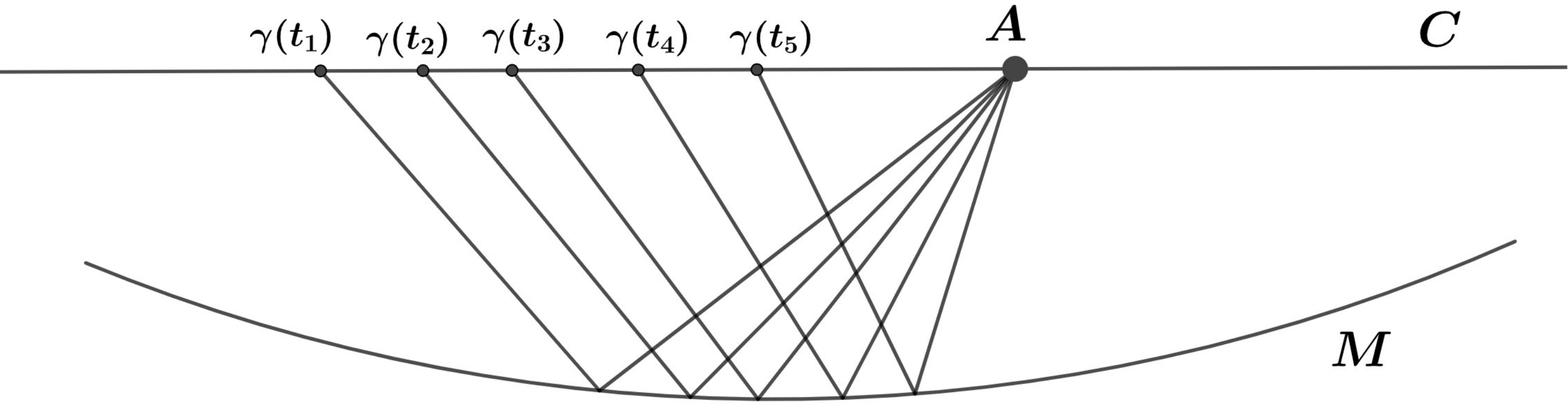}
\caption{Reflection of sound waves.
}
\label{seismic_survey}
\end{center}
\end{figure}
It is known that there exists a curve $W$ parametrized by
$f: I\to \mathbb{R}^2$ with well-defined normals such that
each broken line of a reflected ray starting at $A$ and finishing on $C$
can be replaced by a straight line which is normal to $W$ and of the same
total length.
The curve $W$ is called the \textit{orthotomic} of $M$
relative to $A$ and conversely the curve $M$ is called the
\textit{anti-orthotomic} of $W$ relative to $A$.
Then, an envelope created by the circle family
\[
\left\{\left.(x, y)\in \mathbb{R}^2\, \right|\,
||(x, y) - \gamma(t)||=||f(t)-\gamma(t) ||\right\}_{t\in I}
\]
recovers $W$ (see Figure 2).
\begin{figure}[h]
\begin{center}
\includegraphics[width=10cm]
{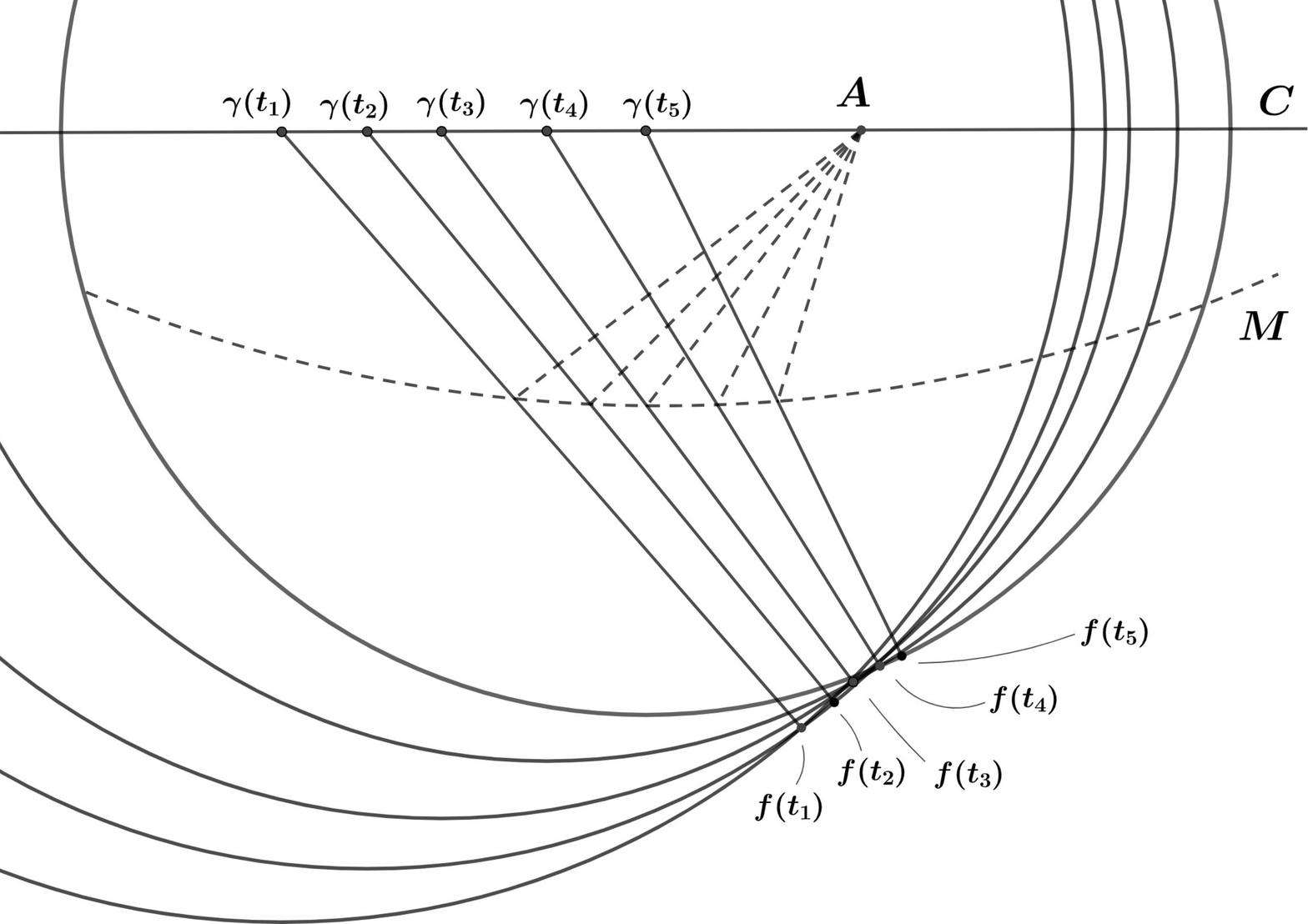}
\caption{An envelope created by the circle family.
}
\label{seismic_survey}
\end{center}
\end{figure}
After obtaining the parametrization $f$ of $W$,
the parametrization $\widetilde{f}$ of $M$ can be easily obtained by using
the anti-orthotomic technique developed in \cite{janeczkonishimura}.
Therefore, in order to investigate the parametrization of $W$
as precisely as possible, 
it is very important to 
construct a general theory on envelopes created by circle families in the plane, 
which is the main purpose of this paper.   
\par
\smallskip
For a point $P$ of $\mathbb{R}^2$
and a positive number $\lambda$,
the circle $C_{(P, \lambda)}$ centered at $P$ with radius $\lambda$
is naturally defined as follows, where the dot in the center stands for
the standard scalar product.
\[
C_{(P, \lambda)}=
\left\{\left.
(x, y)\in \mathbb{R}^2\, \right|\,
\left((x, y) - P\right)\cdot \left((x, y) - P\right) = \lambda^2
\right\}.
\]
For a curve $\gamma: I\to \mathbb{R}^2$ and a positive function
$\lambda: I\to \mathbb{R}_+$,
the circle family $\mathcal{C}_{(\gamma, \lambda)}$
is naturally defined as follows.
Here, $\mathbb{R}_+$ stands for the set consisting of positive real numbers.
\[
\mathcal{C}_{(\gamma, \lambda)}=
\left\{C_{(\gamma(t), \lambda(t))}
\right\}_{t\in I}.
\]
It is reasonable to assume that at each point $\gamma(t)$
the normal vector to the curve $\gamma$ is well-defined.
Thus, we easily reach the following definition.
\begin{definition}\label{definition_frontal}
{\rm
A curve $\gamma: I\to \mathbb{R}^2$ is called a \textit{frontal}
if there exists a mapping $\nu: I\to S^1$ such that
the following identity holds for each $t\in I$, where $S^1$ is the unit circle in
$\mathbb{R}^2$.
\[
\frac{d \gamma}{dt}(t)\cdot \nu(t) =0.
\]
For a frontal $\gamma$,
the mapping $\nu: I\to S^1$ given above is called the
\textit{Gauss mapping} of $\gamma$.
}
\end{definition}
\noindent
By definition, a frontal is a mapping giving a solution 
of the first order linear differential equation
defined by Gauss mapping $\nu$.
Thus, for a fixed mapping $\nu: I\to S^1$
the set consisting of frontals with a given
Gauss mapping $\nu: I\to S^1$ is a linear space.
For frontals, \cite{ishikawa} is recommended as an excellent reference.
Hereafter in this paper, the curve $\gamma: I\to \mathbb{R}^2$ for
a circle family $\mathcal{C}_{(\gamma, \lambda)}$ is assumed to be a frontal.
\par
In this paper,
the following is adopted as the definition of an envelope created
by a circle family.
\begin{definition}\label{definition_envelope}
{\rm
Let $\mathcal{C}_{(\gamma, \lambda)}$ be a circle family.
A mapping $f: I\to \mathbb{R}^2$ is called an
\textit{envelope} created by $\mathcal{C}_{(\gamma, \lambda)}$
if 
the following two hold for any $t\in I$.       
\begin{enumerate}
\item[(1)] $\frac{d f}{d t}(t)\cdot \left(f(t)-\gamma(t)\right)=0$.
\item[(2)] $f(t)\in C_{(\gamma(t),  \lambda(t))}$.
\end{enumerate}
}
\end{definition}
\noindent
By definition,
as same as an envelope created by a hyperplane family (see \cite{nishimura}),
an envelope created by a circle family is a mapping giving a solution 
of a first order differential equation with one constraint condition.
Moreover, again by definition, an envelope $f$ created by a circle family 
$\mathcal{C}_{(\gamma, \lambda)}$ is a frontal 
with Gauss mapping $I\ni t\to \frac{f(t)-\gamma(t)}{||f(t)-\gamma(t)||}\in S^1$.
\begin{problem}\label{problem1}   
Let $\gamma: I\to \mathbb{R}^2$ be a frontal with Gauss mapping 
$\nu: I\to S^1$ and let $\lambda: I\to \mathbb{R}_+$ be a positve function.    
\begin{enumerate}
\item[(1)] 
Find a necessary and sufficient condition for
the circle family $\mathcal{C}_{(\gamma, \lambda)}$ to create an envelope
in terms of $\gamma$, $\nu$ and $\lambda$.
\item[(2)] Suppose that the circle family $\mathcal{C}_{(\gamma, \lambda)}$
creates an envelope.    Then,
find a parametrization of the envelope
created by $\mathcal{C}_{(\gamma, \lambda)}$ 
in terms of $\gamma$, $\nu$ and $\lambda$.   
\item[(3)] Suppose that the circle family $\mathcal{C}_{(\gamma, \lambda)}$
creates an envelope.    Then, 
find a criterion for the number of distinct envelopes 
created by $\mathcal{C}_{(\gamma, \lambda)}$
in terms of $\gamma$, $\nu$ and $\lambda$.
\end{enumerate}
\end{problem}
\begin{note}\label{note1}
\item[(1)] (1) of Problem \ref{problem1} is a problem to seek the
integrability conditions.  There are various cases, for instance
the concentric circle family
$\{\{(x, y)\in \mathbb{R}^2\, |\, x^2+y^2=t^2\}\}_{t\in \mathbb{R}_+}$
does not create an envelope while the parallel-translated circle family
$\{\{(x, y)\in \mathbb{R}^2\, |\, (x-t)^2+y^2=1\}\}_{t\in \mathbb{R}}$
does create two envelopes.   Thus, (1) of Problem \ref{problem1} is significant.
\item[(2)] The following Example \ref{example1} shows that
the well-known method to represent the envelope is useless in this case.
Thus, (2) of Problem \ref{problem1} is important and the positive answer to it is
much desired.
\item[(3)]  The following Example \ref{example2} shows that
there are at least three cases: the case having a unique envelope, the case
having exactly two envelopes and the case having uncountably many envelopes.
Thus, (3) of Problem \ref{problem1} is meaningful and interesting.
\end{note}
\begin{example}\label{example1}
Let $\gamma: \mathbb{R}\to \mathbb{R}^2$ be the mapping defined by
$\gamma(t)=\left(t^3, t^6\right)$.     Set
$\nu(t)=\frac{1}{\sqrt{4t^6+1}}\left(-2t^3, 1\right)$.
It is clear that the mapping $\gamma$ is a frontal with Gauss mapping
$\nu: \mathbb{R}\to S^1$.
Let $\lambda: \mathbb{R}\to \mathbb{R}_+$ be the constant function
defined by $\lambda(t)=1$.
Then, it seems that the circle family
$\mathcal{C}_{(\gamma, \lambda)}$ creates envelopes.
Thus, we can expect that the created envelopes can be obtained
by the well-known method.
Set $F(x, y, t)=\left(x-t^3\right)^2+\left(y-t^6\right)^2-1$.
Then, we have the following.
{\small
\begin{eqnarray*}
\mathcal{D} & = &
\left\{(x, y)\in \mathbb{R}^2\, \left|\,
\exists t \mbox{ such that }F(x,y,t)=\frac{\partial F}{\partial t}(x, y, t)=0
\right.\right\} \\
{ } & = &
\left\{(x, y)\in \mathbb{R}^2\, \left|\,
\exists t \mbox{ such that }
\left(x-t^3\right)^2+\left(y-t^6\right)^2-1=0, \,
-6t^2\left(x-t^3\right)-12t^5\left(y-t^6\right)=0
\right.\right\} \\
{ } & = &
\left\{(x, y)\in \mathbb{R}^2\, \left|\,
\exists t \mbox{ such that }
\left(x-t^3\right)^2+\left(y-t^6\right)^2-1=0, \,
t^2\left(\left(x-t^3\right)+2t^3\left(y-t^6\right)\right)=0
\right.\right\} \\
{ } & = &
\left\{(x, y)\in \mathbb{R}^2\, \left|\,
x^2+y^2=1\right.\right\} \bigcup
\left\{(x, y)\in \mathbb{R}^2\, \left|\,
\left(x-t^3\right)^2+\left(y-t^6\right)^2-1=0, \,
x=t^3-2t^3\left(y-t^6\right)
\right.\right\} \\
{ } & = &
\left\{(x, y)\in \mathbb{R}^2\, \left|\,
x^2+y^2=1\right.\right\} \bigcup
\left\{(x, y)\in \mathbb{R}^2\, \left|\,
\left(-2t^3\left(y-t^6\right)\right)^2+\left(y-t^6\right)^2=1, \,
x=t^3\left(1-2y+2t^6\right)
\right.\right\}
\\
{ } & = &
\left\{(x, y)\in \mathbb{R}^2\, \left|\,
x^2+y^2=1\right.\right\} \bigcup
\left\{\left.
\left(t^3\mp \frac{2t^3}{\sqrt{4t^6+1}}, t^6\pm \frac{1}{\sqrt{4t^6+1}}\right)
\in \mathbb{R}^2\, \right|\, t\in \mathbb{R}
\right\}.
\end{eqnarray*}
}
In Example \ref{example3} of Section \ref{section3},
it turns out that the set $\mathcal{D}$ calculated here is actually larger than
the set of envelopes created by $\mathcal{C}_{(\gamma, \lambda)}$,
namely the unit circle $\left\{(x, y)\in \mathbb{R}^2\, \left|\,
x^2+y^2=1\right.\right\}$ is redundant.
Therefore, 
the well-known method 
to represent the envelopes does not 
work well in general and it must be applied under 
appropriate assumptions even for circle families.     
The circle family $\mathcal{C}_{(\gamma, \lambda)}$ 
and the candidates of its envelope are depicted in
Figure \ref{figure_example1}.
\begin{figure}[h]
\begin{center}
\includegraphics[width=10cm]
{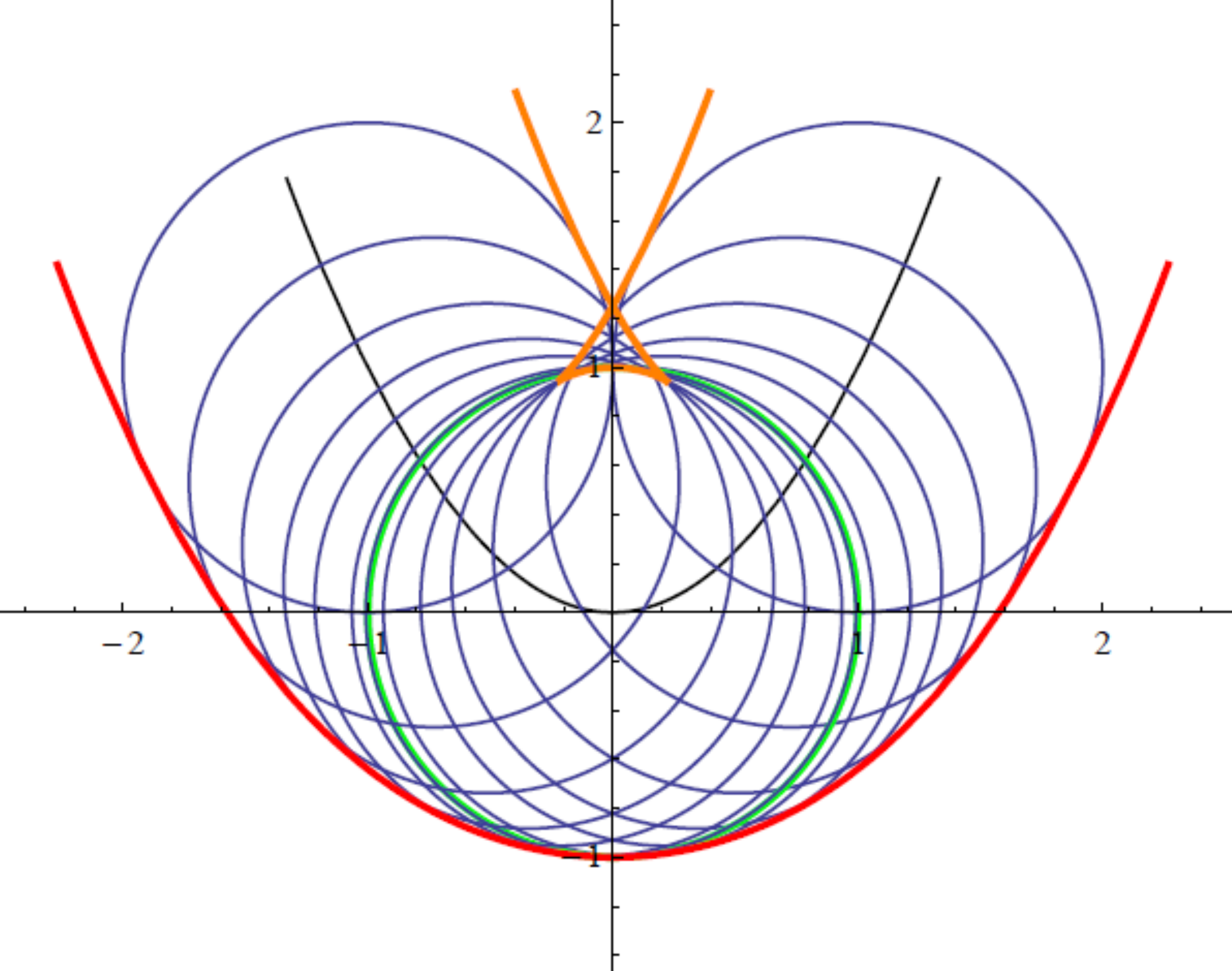}
\caption{The circle family $\mathcal{C}_{(\gamma, \lambda)}$
and the candidates of its envelope.
}
\label{figure_example1}
\end{center}
\end{figure}
\end{example}
\begin{example}\label{example2}
\begin{enumerate}
\item[(1)] Let $\gamma: \mathbb{R}_+\to \mathbb{R}^2$
be the mapping defined by
$\gamma(t)=\left(0, 1+t\right)$.     Then, it is clear that $\gamma$ is a frontal.
Let $\lambda: \mathbb{R}_+\to \mathbb{R}_+$ be the positive function
defined by $\lambda(t)=1+t$.
Then, it is easily seen that the origin $(0,0)$ of the plane $\mathbb{R}^2$
itself is a created envelope by the circle family
$\mathcal{C}_{(\gamma, \lambda)}$
and that there are no other envelopes created by
$\mathcal{C}_{(\gamma, \lambda)}$.
Hence, the number of created envelopes is one
in this case.
\item[(2)] The parallel-translated circle family
$\{\{(x, y)\in \mathbb{R}^2\, |\, (x-t)^2+y^2=1\}\}_{t\in \mathbb{R}}$
creates exactly two envelopes.
\item[(3)]   Let $\gamma: \mathbb{R}\to \mathbb{R}^2$ be the constant
mapping defined by $\gamma(t)=\left(0, 0\right)$.
Then, it is clear that $\gamma$ is a frontal.
Let $\lambda: \mathbb{R}\to \mathbb{R}_+$ be the constant function
defined by $\lambda(t)=1$.
Then, for any function $\theta: \mathbb{R}\to \mathbb{R}$,
the mapping $f: \mathbb{R}\to \mathbb{R}^2$ defined by
$f(t)=(\cos\theta(t), \sin\theta(t))$ is an envelope created by
the circle family $\mathcal{C}_{(\gamma, \lambda)}$.
Hence, there  are uncountably many created envelopes in this case.
\end{enumerate}
\end{example}
\par
In order to solve Problem \ref{problem1},
we prepare several terminologies which can be derived
from a frontal $\gamma: I\to \mathbb{R}^2$ with
Gauss mapping $\nu: I\to S^1$ and
a positive function $\lambda: I\to \mathbb{R}_+$.
For a frontal $\gamma: I\to \mathbb{R}^2$ with Gauss mapping
$\nu: I\to S^1$,
following \cite{fukunagatakahashi},
we set $\mu(t) = J(\nu(t))$, where $J$ is the anti-clockwise
rotation by $\pi/2$.
Then we have a moving frame $\left\{\mu(t), \nu(t)\right\}_{t\in I}$ along
the frontal $\gamma$.
Set
\[
\ell (t)=\frac{d \nu}{dt}(t)\cdot \mu(t), \quad
\beta(t)=\frac{d \gamma}{dt}(t)\cdot \mu(t).
\]
The pair of functions $(\ell, \beta)$ is called the \textit{curvature} of the
frontal $\gamma$ with Gauss mapping $\nu$.
We want to focus the ratio of $\frac{d \lambda}{dt}(t)$ and $\beta(t)$.
The following definition is the key of this paper.
\begin{definition}\label{creative}
{\rm
Let $\gamma: I\to \mathbb{R}^2$,  $\lambda: I\to \mathbb{R}_+$
be a frontal with Gauss mapping $\nu: I\to S^1$ and
a positive function respectively.
Then, the circle family $\mathcal{C}_{(\gamma, \lambda)}$ is said to be
\textit{creative} if there exists a mapping
$\widetilde{\nu}: I\to S^1$ such that the following identity holds for any
$t\in I$.
\[
\frac{d\lambda}{dt}(t)=-\beta(t)\left(\widetilde{\nu}(t)\cdot \mu(t)\right).
\]
Set $\cos\theta(t)=-\widetilde{\nu}(t)\cdot \mu(t)$.    Then, the creative
condition is equivalent to say that
there exists a function $\theta: I\to \mathbb{R}$ satisfying the following identity
for any $t\in I$.
\[
\frac{d\lambda}{dt}(t)=\beta(t)\cos\theta(t).
\]
}
\end{definition}
\noindent
By definition, any family of concentric circles
with smoothly expanding radii is not creative, 
and it is clear that such the circle family
does not create an envelope.
Under the above preparation, Problem \ref{problem1} is solved
as follows.
\begin{theorem}\label{theorem1}
Let $\gamma: I\to \mathbb{R}^2$
be a frontal with Gauss mapping $\nu: I\to S^1$ and let
$\lambda: I\to \mathbb{R}_+$ be a positive function.
Then, the following three hold.     
\begin{enumerate}
\item[(1)] The circle family $\mathcal{C}_{(\gamma, \lambda)}$
creates an envelope if and only if $\mathcal{C}_{(\gamma, \lambda)}$ is creative.
\item[(2)] Suppose that the circle family
$\mathcal{C}_{(\gamma, \lambda)}$ creates
an envelope $f: I\to \mathbb{R}^2$.     Then, the created envelope
$f$ is represented as follows.
\[
f(t) = \gamma(t) + \lambda(t)\widetilde{\nu}(t), 
\]
where $\widetilde{\nu}: I\to S^1$ is the mapping
defined in Definition \ref{creative}.
\item[(3)] Suppose that the circle family
$\mathcal{C}_{(\gamma, \lambda)}$ creates
an envelope.
Then, the number of envelopes created by $\mathcal{C}_{(\gamma, \lambda)}$ is
characterized as follows.
\smallskip
\begin{enumerate}
\item[(3-i)] The circle family
$\mathcal{C}_{(\gamma, \lambda)}$
creates a unique envelope
if and only if
the set consisting of $t\in I$ satisfying $\beta(t)\ne 0$  and
$\frac{d\lambda}{dt}(t)=\pm \beta(t)$ is dense in $I$.
\item[(3-ii)] There are exactly two distinct envelopes created by
$\mathcal{C}_{(\gamma, \lambda)}$ if and only if
the set of $t\in I$ satisfying $\beta(t)\ne 0$ is dense in $I$ and
there exists at least one $t_0\in I$ such that the strict inequality
$|\frac{d\lambda}{dt}(t_0)| < |\beta(t_0)|$ holds.
\item[(3-$\infty$)] There are uncountably many distinct
envelopes created by
$\mathcal{C}_{(\gamma, \lambda)}$ if and only if
the set of $t\in I$ satisfying $\beta(t)\ne 0$ is not dense in $I$.
\end{enumerate}
\end{enumerate}
\end{theorem}
\noindent
By the assertion (2) of Theorem \ref{theorem1}, it is reasonable to call
$\widetilde{\nu}\,$ the \textit{creator} for an envelope $f$ created by
$\mathcal{C}_{(\gamma, \lambda)}$.
\bigskip
\par
This paper is organized as follows.
Theorem \ref{theorem1} is proved in Section \ref{section2}.
In Section \ref{section3},
several examples to which
Theorem \ref{theorem1} 
is effectively applicable are given.
Finally, in Section 4, relations of several definitions of an envelope created by
a circle family are investigated.
\section{Proof of Theorem \ref{theorem1}}\label{section2}
\subsection{Proof of the assertion (1) of Theorem \ref{theorem1}}
\label{subsection2.1}
Suppose that $\mathcal{C}_{(\gamma, \lambda)}$ is creative.
By definition, there exists a mapping $\widetilde{\nu}: I\to S^1$
such that the equality
$\frac{d \lambda}{dt}(t)=
-\beta(t)\left(\widetilde{\nu}(t)\cdot \mu(t)\right)$ holds for any $t\in I$.
Set
\[
f(t) = \gamma(t)+\lambda(t)\widetilde{\nu}(t).
\]
Then, since
$\left(f(t)-\gamma(t)\right)\cdot \left(f(t)-\gamma(t)\right)=\lambda^2(t)$,
it follows $f(t)\in C_{(\gamma(t), \lambda(t))}$.
Moreover, since
\[
\frac{d f}{dt}(t)=\frac{d \gamma}{dt}(t) +
\frac{d \lambda}{dt}(t)\widetilde{\nu}(t)
+\lambda(t)\frac{d \widetilde{\nu}}{dt}(t),
\]
we have the following.
\begin{eqnarray*}
{ } & { } & \frac{d f}{dt}(t)\cdot \left(f(t)-\gamma(t)\right) \\
{ } & = & \left(\frac{d \gamma}{dt}(t) +
\frac{d \lambda}{dt}(t)\widetilde{\nu}(t)
+\lambda(t)\frac{d \widetilde{\nu}}{dt}(t)\right)\cdot
\left(\lambda(t)\widetilde{\nu}(t)\right) \\
{ } & = & \frac{d \gamma}{dt}(t)\cdot \left(\lambda(t)\widetilde{\nu}(t)\right)
+
\frac{d \lambda}{dt}(t)\lambda(t) \\
{ } & = &
\left(\beta(t)\mu(t)\right)\cdot \left(\lambda(t)\widetilde{\nu}(t)\right)
+
\left(-\beta(t)\left(\widetilde{\nu}(t)\cdot \mu(t)\right)\right)\lambda(t) \\
{ } & = &
\beta(t)\lambda(t)\left(\mu(t)\cdot \widetilde{\nu}(t)\right)
-
\beta(t)\lambda(t)\left(\widetilde{\nu}(t)\cdot \mu(t)\right) \\
{ } & = & 0.
\end{eqnarray*}
Hence, $f$ is an envelope created by the circle family
$\mathcal{C}_{(\gamma, \lambda)}$.
\par
Conversely, suppose that the circle family
$\mathcal{C}_{(\gamma, \lambda)}$ creates an envelope $f: I\to \mathbb{R}$.
Then, by definition, it follows that
$f(t)\in C_{(\gamma(t), \lambda(t))}$ and
$\frac{d f}{d t}(t)\cdot \left(f(t)-\gamma(t)\right)=0$.
The condition $f(t)\in C_{(\gamma(t), \lambda(t))}$ implies
that there exists a mapping $\widetilde{\nu}: I\to S^1$ such that
the following equality holds for any $t\in I$.
\[
f(t)=\gamma(t)+\lambda(t)\widetilde{\nu}(t).
\]
Then, since
\[
\frac{d f}{dt}(t)=\frac{d \gamma}{dt}(t) +
\frac{d \lambda}{dt}(t)\widetilde{\nu}(t)
+\lambda(t)\frac{d \widetilde{\nu}}{dt}(t),
\]
we have the following.
\begin{eqnarray*}
0 & = & \frac{df}{dt}(t)\cdot \left(f(t)-\gamma(t)\right) \\
{ } & = &
\left(\frac{d \gamma}{dt}(t) +
\frac{d \lambda}{dt}(t)\widetilde{\nu}(t)
+\lambda(t)\frac{d \widetilde{\nu}}{dt}(t)\right)
\cdot
\left(\lambda(t)\widetilde{\nu}(t)\right) \\
{ } & = &
\left(\beta(t)\mu(t)\right)\cdot \left(\lambda(t)\widetilde{\nu}(t)\right)
+ \frac{d \lambda}{dt}(t)\lambda(t) \\
{ } & = &
\lambda(t)\left(\beta(t)\left(\mu(t)\cdot \widetilde{\nu}(t)\right)
+ \frac{d\lambda}{dt}(t)\right).
\end{eqnarray*}
Since $\lambda(t)$ is positive for any $t\in I$, it follows
\[
\beta(t)\left(\mu(t)\cdot \widetilde{\nu}(t)\right)
+ \frac{d\lambda}{dt}(t)=0.
\]
Therefore, the circle family $\mathcal{C}_{(\gamma, \lambda)}$ is creative.
\hfill $\Box$
\par
\medskip
\noindent
\subsection{Proof of the assertion (2) of Theorem \ref{theorem1}}
\label{subsection2.2}
The proof of the assertion (1) given in Subsection \ref{subsection2.1} proves
the assertion (2) as well.
\hfill $\Box$
\par
\medskip
\noindent
\subsection{Proof of the assertion (3) of Theorem \ref{theorem1}}
\label{subsection2.3}
\subsubsection{Proof of (3-i)}
Suppose that the circle family $\mathcal{C}_{(\gamma, \lambda)}$
creates a unique envelope.
Then, for any $t\in I$ the unit vector $\widetilde{\nu}(t)$ satisfying
\[
\frac{d\lambda}{dt}(t)=-\beta(t)\left(\widetilde{\nu}(t)\cdot\mu(t)\right)
\]
must be uniquely determined.     Hence, under considering continuity
of two functions $\frac{d\lambda}{dt}$ and $\beta$, it follows that
the set consisting of $t\in I$ satisfying
$\frac{d\lambda}{dt}(t)=\pm \beta(t)\ne 0$ must be dense in $I$.
\par
Conversely,
suppose that the set consisting of $t\in I$ satisfying
$\frac{d\lambda}{dt}(t)=\pm \beta(t)\ne 0$ is dense in $I$.
Then,  under considering continuity of
the function $t\mapsto \widetilde{\nu}(t)\cdot\mu(t)$, it follows that
$\widetilde{\nu}(t)\cdot\mu(t)=\pm 1$ for any $t\in I$.
Thus, the created envelope
$f(t)=\gamma(t)+\lambda(t)\widetilde{\nu}(t)$ must be unique.
\hfill $\Box$
\par
\smallskip
\noindent
\subsubsection{Proof of (3-ii)}
Suppose that there are exactly two distinct envelopes created by
$\mathcal{C}_{(\gamma, \lambda)}$.
Then, by the equality
$
\frac{d\lambda}{dt}(t)=-\beta(t)\left(\widetilde{\nu}(t)\cdot\mu(t)\right),
$
the set consisting of $t\in I$ satisfying
$\beta(t)\ne 0$ must be dense in $I$.
Suppose moreover that
the equality $\frac{d\lambda}{dt}(t)=\pm \beta(t)$ holds for any $t\in I$.
Then, it follows that the set consisting of $t\in I$ satisfying
$\frac{d\lambda}{dt}(t)=\pm \beta(t)\ne 0$ is dense in $I$.
Then, by the assertion (3-i), the given circle family must create a unique
envelope.     This contradicts the assumption that
there are exactly two distinct envelopes.
Hence, there must exist at least one $t_0\in I$ such that the strict inequality
 $|\frac{d\lambda}{dt}(t_0)|<|\beta(t_0)|$ holds.
\par
Conversely, suppose that
the set of $t\in I$ satisfying $\beta(t)\ne 0$ is dense in $I$ and
there exists at least one $t_0\in I$ such that the strict inequality
$|\frac{d\lambda}{dt}(t_0)| < |\beta(t_0)|$ holds.
Then, it follows that there must exist an open interval
$\widetilde{I}$ in $I$ such that  the absolute value
$|\widetilde{\nu}(t)\cdot\mu(t)|=|\cos\theta(t)|$
is less than $1$ for any
$t\in \widetilde{I}$.  Thus, it follows
$\theta(t)\ne -\theta(t)$
for any  $t\in \widetilde{I}$.
Hence, for any  $t\in \widetilde{I}$,
there exist exactly two distinct unit vectors
$\widetilde{\nu}_+(t), \widetilde{\nu}_-(t)$  corresponding
$\widetilde{\nu}_+(t)\cdot\mu(t)=-\cos\theta(t)$ and
$\widetilde{\nu}_-(t)\cdot\mu(t)=-\cos\left(-\theta(t)\right)$
respectively.
Therefore, by the assertion (2) of Theorem \ref{theorem1},  
the circle family must create
exactly two distinct envelopes.
\hfill $\Box$
\par
\smallskip
\noindent
\subsubsection{Proof of (3-$\infty$)}
Suppose that there are uncountably many distinct
envelopes created by
$\mathcal{C}_{(\gamma, \lambda)}$.
Suppose moreover that the set of $t\in I$ such that
$\beta(t)\ne 0$ is dense in $I$.   Then, from
(3-i) and (3-ii), it follows that
the circle family $\mathcal{C}_{(\gamma, \lambda)}$ must create
a unique envelope or two distinct envelopes.
This contradicts the assumption that  there are uncountably many distinct
envelopes created by
$\mathcal{C}_{(\gamma, \lambda)}$.
Hence, the set of $t\in I$ such that
$\beta(t)\ne 0$ is never dense in $I$.
\par
Conversely, suppose that
the set of $t\in I$ such that
$\beta(t)\ne 0$ is not dense in $I$.
This assumption implies that there exists an open interval
$\widetilde{I}$ in $I$ such that $\beta(t)=0$ for any
$t\in \widetilde{I}$.
On the other hand, since $\mathcal{C}_{(\gamma, \lambda)}$ creates
an envelope $f_0$, the equality
\[
\frac{d\lambda}{dt}(t)=-\beta(t)\left(\widetilde{\nu}(t)\cdot\mu(t)\right)
\]
holds for any $t\in I$.
Thus, there are no restrictions for the value
$\widetilde{\nu}(t)\cdot \mu(t)$ for any $t\in \widetilde{I}$.
Take one point $t_0$ of $\widetilde{I}$ and denote
the $\widetilde{\nu}$ for the envelope $f_0$ by
$\widetilde{\nu}_0$.
Then, by using the standard technique on bump functions,
we may construct uncountably many distinct creators
$\widetilde{\nu}_a : I\to S^1$ $(a\in A)$
such that the following (a), (b), (c) and (d) hold,
where $A$ is a set consisting uncountably many elements such that
$0\not\in A$.
\begin{enumerate}
\item[(a)] The equality
$\frac{d\lambda}{dt}(t)=
-\beta(t)\left(\widetilde{\nu}_a(t)\cdot\mu(t)\right)$
holds for any $t\in I$ and any $a\in A$.
\item[(b)] For any $t\in I-\widetilde{I}$ and any $a\in A$,
the equality
$\widetilde{\nu}_a(t)=\widetilde{\nu}_0(t)$ holds.
\item[(c)] For any $a\in A$, the property
$\widetilde{\nu}_a(t_0)\ne \widetilde{\nu}_0(t_0)$ holds.
\item[(d)] For any two distinct $a_1, a_2\in A$, the property
$\widetilde{\nu}_{a_1}(t_0)\ne
\widetilde{\nu}_{a_2}(t_0)$ holds.
\end{enumerate}
Therefore, by the assertion (2) of Theorem \ref{theorem1},  
the circle family  $\mathcal{C}_{(\gamma, \lambda)}$ creates
uncountably many distinct envelopes.
\hfill $\Box$
\par
\smallskip
\noindent
\section{Examples}\label{section3}
\begin{example}\label{example3}
We examine Example \ref{example1} by applying Theorem \ref{theorem1}.
In Example \ref{example1}, $\gamma: \mathbb{R}\to \mathbb{R}^2$
is given by $\gamma(t)=\left(t^3, t^6\right)$.     Thus, we can say that
$\nu: \mathbb{R}\to S^1$ and $\mu: \mathbb{R}\to S^1$ are given by
$\nu(t)=\frac{1}{\sqrt{4t^6+1}}\left(-2t^3, 1\right)$ and
$\mu(t)=\frac{1}{\sqrt{4t^6+1}}\left(-1, -2t^3\right)$ respectively.
Moreover, the radius function
$\lambda: \mathbb{R}\to \mathbb{R}$ is the constant
function defined by $\lambda(t)=1$.
Thus,
\[
\frac{d \lambda}{dt}(t)=0.
\]
By calculation, we have
\[
\beta(t)=\frac{d \gamma}{d t}(t)\cdot \mu(t) =\frac{-3t^2(1+4t^6)}{\sqrt{4t^6+1}}.
\]
Therefore, the unit vector $\widetilde{\nu}(t)\in S^1$ satisfying
\[
\frac{d\lambda}{d t}(t) =
-\beta(t)\left(\widetilde{\nu}(t)\cdot \mu(t)\right)
\]
exists and it must have the form
\[
\widetilde{\nu}(t)=\pm \nu(t)=\frac{\pm 1}{\sqrt{4t^6+1}}\left(-2t^3, 1\right).
\]
Hence, by the assertion (1) of Theorem \ref{theorem1}, the circle family
$\mathcal{C}_{(\gamma, \lambda)}$ creates an envelope
$f: \mathbb{R}\to \mathbb{R}^2$.
By the assertion (2) of Theorem \ref{theorem1}, $f$ is parametrized as follows.
\begin{eqnarray*}
f(t) & = & \gamma(t)+\lambda(t)\widetilde{\nu}(t) \\
{ } & = & \left(t^3, t^6\right) \pm \frac{1}{\sqrt{4t^6+1}}\left(-2t^3, 1\right) \\
{ } & = &
\left(
t^3\mp \frac{2t^3}{\sqrt{4t^6+1}},\, t^6\pm \frac{1}{\sqrt{4t^6+1}}
\right).
\end{eqnarray*}
Finally, by the assertion (3-ii) of Theorem \ref{theorem1},
the number of distinct envelopes created by the
circle family $\mathcal{C}_{(\gamma, \lambda)}$ is exactly two.
\par
Therefore, Theorem \ref{theorem1} reveals that
the set $\mathcal{D}$ calculated in Example \ref{example1} is
certainly the union of the unit circle and the set of two envelopes of
$\mathcal{C}_{(\gamma, \lambda)}$.
\end{example}
\begin{example}\label{example4}
We examine (1) of Example \ref{example2} by applying Theorem \ref{theorem1}.
In (1) of Example \ref{example2}, $\gamma: \mathbb{R}_+\to \mathbb{R}^2$
is given by $\gamma(t)=\left(0, 1+t\right)$.
Thus, if we define the unit vector $\nu(t)=(1, 0)$,
$\nu: \mathbb{R}_+\to S^1$ gives the Gauss mapping of $\gamma$.
By definition, $\mu(t)=(0,1)$ and thus we have
$\beta(t)=\frac{d\gamma}{d t}(t)\cdot \mu(t)=1$.
On the other hand, the radius function
$\lambda: \mathbb{R}_+\to \mathbb{R}_+$ has the form
$\lambda(t)=1+t$ in this example.
Thus, the creative condition
\[
\frac{d\lambda}{dt}(t) =
-\beta(t)\left(\widetilde{\nu}(t)\cdot \mu(t)\right)
\]
simply becomes 
\[
1 =
-\left(\widetilde{\nu}(t)\cdot (0, 1)\right)
\leqno{(*)}
\]
in this case.
If we set $\widetilde{\nu}(t)=(0,-1)$, then the above equality holds for any
$t\in \mathbb{R}_+$.     
Thus, by the assertion (1) of Theorem \ref{theorem1}, the circle 
family $\mathcal{C}_{(\gamma, \lambda)}$ creates an envelope.
By the assertion (2) of Theorem \ref{theorem1},
the parametrization of the created envelope is
\[
f(t)=\gamma(t)+\lambda(t)\widetilde{\nu}(t)
= \left(0, 1+t\right)+\left(1+t\right)\left(0, -1\right) =\left(0, 0\right).
\]
Finally, notice that for any
$t\in \mathbb{R}_+$
the creative condition (*) in this case
holds if and only if $\widetilde{\nu}(t)=(0,-1)=-\mu(t)$.
Thus, by the assertion (3-i) of Theorem \ref{theorem1},
the origin $(0,0)$  is the unique envelope created by
$\mathcal{C}_{(\gamma, \lambda)}$.
\end{example}
\begin{example}\label{example5}
Theorem \ref{theorem1} can be applied also
to (2) of Example \ref{example2} as follows.
In this example,
$\gamma(t)=(t, 0)$ and $\lambda(t)=1$.
Thus, we may set $\nu(t)=(0, -1)$, $\mu(t)=(1, 0)$.
It follows $\beta(t)=\frac{d\gamma}{d t}(t)\cdot \mu(t)=1$.
Since the radius function $\lambda$ is a constant function,
the creative condition
\[
\frac{d\lambda}{dt}(t) =
-\beta(t)\left(\widetilde{\nu}(t)\cdot \mu(t)\right)
\]
simply becomes 
\[
0 =
-\left(\widetilde{\nu}(t)\cdot (0, 1)\right)
\]
in this case.
Thus, for any $t\in \mathbb{R}$,
the creative condition is satisfied if and only if
$\widetilde{\nu}(t)=\pm (1,0)$.
Hence, by the assertion (1) of Theorem \ref{theorem1}, the circle
family $\mathcal{C}_{(\gamma, \lambda)}$ creates an envelope.
By the assertion (2) of Theorem \ref{theorem1},
the parametrization of the created envelope is
\[
f(t)=\gamma(t)+\lambda(t)\widetilde{\nu}(t)
= \left(t, 0\right)\pm\left(0, -1\right) =\left(t, \mp 1\right).
\]
Finally, by the assertion (3-ii) of Theorem \ref{theorem1},
the number of envelope created by
$\mathcal{C}_{(\gamma, \lambda)}$ is exactly two.
\end{example}
\begin{example}\label{example6}
Theorem \ref{theorem1} can be applied even
to (3) of Example \ref{example2} as follows.
In this example,
$\gamma(t)=(0, 0)$ and $\lambda(t)=1$.
Thus, every mapping $\nu: \mathbb{R}\to S^1$ can be taken
as Gauss mapping of $\gamma$.   In particular,
$\gamma$ is a frontal.
We have $\beta(t)=\frac{d\gamma}{d t}(t)\cdot \mu(t)=0$.
Since the radius function $\lambda$ is a constant function $\lambda(t)=1$,
the creative condition
\[
\frac{d\lambda}{dt}(t) =
-\beta(t)\left(\widetilde{\nu}(t)\cdot \mu(t)\right)
\] 
simply becomes 
\[
0 =
0
\]
in this case.
Thus, for any $\widetilde{\nu}: \mathbb{R}\to S^1$,
the creative condition is satisfied.
Hence, by the assertion (1) of Theorem \ref{theorem1}, the circle
family $\mathcal{C}_{(\gamma, \lambda)}$ creates an envelope.
By the assertion (2) of Theorem \ref{theorem1},
the parametrization of the created envelope is
\[
f(t)=\gamma(t)+\lambda(t)\widetilde{\nu}(t)
= \left(0, 0\right)+\widetilde{\nu}(t)=\widetilde{\nu}(t).
\]
Finally, by the assertion (3-$\infty$) of Theorem \ref{theorem1},
there are uncountably many distinct envelope created by
$\mathcal{C}_{(\gamma, \lambda)}$.
\end{example}
\begin{example}\label{example7}
Let $\gamma: \mathbb{R}_+\to \mathbb{R}^2$ be the mapping defined by
$\gamma(t)=(t, 0)$ and let $\lambda: \mathbb{R}_+\to \mathbb{R}_+$
be the positive function defined by $\lambda(t)=t^2$.
The circle family $\mathcal{C}_{(\gamma, \lambda)}$
and the candidate of its envelope is depicted in
Figure \ref{figure_example7}.
\begin{figure}[h]
\begin{center}
\includegraphics[width=10cm]
{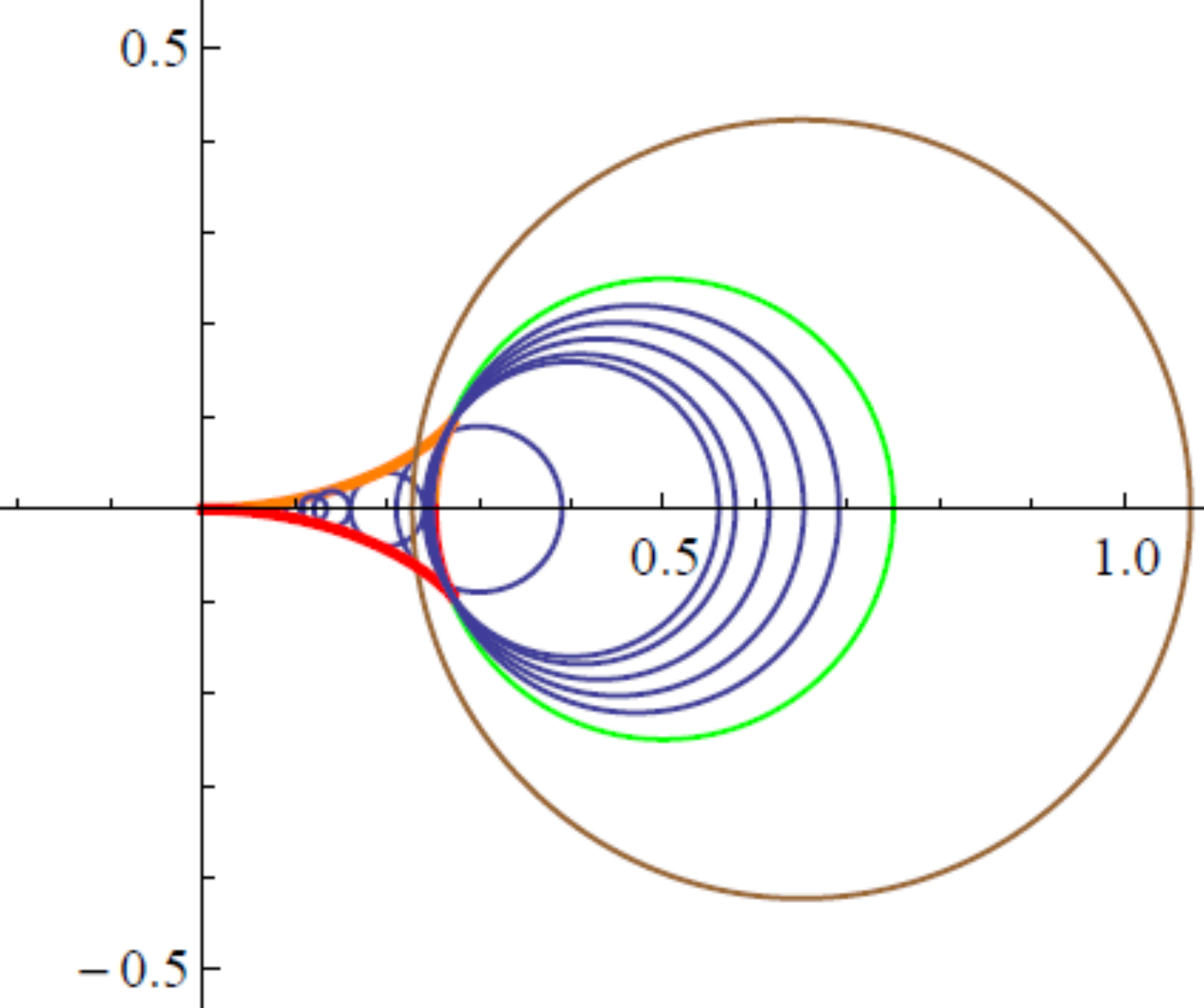}
\caption{The circle family $\mathcal{C}_{(\gamma, \lambda)}$
and the candidate of its envelope.
}
\label{figure_example7}
\end{center}
\end{figure}
Defining the mapping
$\nu: \mathbb{R}_+\to S^1$ by $\nu(t)=(0, -1)$ clarifies that
the mapping $\gamma$ is a frontal.     Then,
$\mu(t)=J(\nu(t))=(1, 0)$ and
$\beta(t)=\frac{d \gamma}{d t}(t)\cdot \mu(t)=(1, 0)\cdot (1, 0)=1$.
We want to seek a mapping $\widetilde{\nu}: \mathbb{R}_+\to S^1$ satisfying
\[
\frac{d\lambda}{dt}(t) = -\beta(t)\left(\widetilde{\nu}(t)\cdot \mu(t)\right),
\]
namely, a mapping $\widetilde{\nu}: \mathbb{R}_+\to S^1$ satisfying
\[
2t=-(\left(\widetilde{\nu}(t)\cdot (1,0)\right)).
\]
Since $\widetilde{\nu}(t)\in S^1$, from the above expression,
it follows that such $\widetilde{\nu}(t)$ does not exist if $\frac{1}{2}<t$.
Thus, the circle family $\mathcal{C}_{(\gamma, \lambda)}$ is not creative
and it creates no envelopes by the assertion 
(1) of Theorem \ref{theorem1}.
\end{example}
\begin{example}\label{example8}
This example is almost the same as Example \ref{example7}.
The difference from Example \ref{example7} is only the parameter space.
In Example \ref{example8}, the parameter space $I$ is
$\left(0, \frac{1}{2}\right)$.
That is to say, in this example, $\mathbb{R}_+$ in Example \ref{example7}
is replaced by $\left(0, \frac{1}{2}\right)$ and all other settings in
Example \ref{example7} remain without change.
\par
Then, from calculations in Example \ref{example7}, it follows that
the given circle family $\mathcal{C}_{(\gamma, \lambda)}$ is creative.
Thus, by the assertion (1) of Theorem \ref{theorem1},
$\mathcal{C}_{(\gamma, \lambda)}$ creates an envelope.
It is easily seen that the expression of $\widetilde{\nu}(t)$ must be
as follows.
\[
\widetilde{\nu}(t)=\left(-2t, \pm\sqrt{1-4t^2}\right).
\]
Therefore, by the assertion (2) of Theorem \ref{theorem1}, an envelope $f$
created by $\mathcal{C}_{(\gamma, \lambda)}$ is
parametrized as follows.
\begin{eqnarray*}
f(t) & = & \gamma(t)+\lambda(t)\widetilde{\nu}(t) \\
{ } & = &
(t, 0)+t^2\left(-2t,\, \pm\sqrt{1-4t^2}\right)  \\
{ } & = &
\left(t-2t^3,\, \pm t^2\sqrt{1-4t^2}\right).
\end{eqnarray*}
Finally, by the assertion (3-ii) of Theorem \ref{theorem1}, it follows that
the number of distinct envelopes created by the
circle family $\mathcal{C}_{(\gamma, \lambda)}$ is exactly two.
\end{example}
\begin{example}\label{example9}
Let $\gamma: \mathbb{R}\to \mathbb{R}^2$ be the mapping defined by
$\gamma(t)=(t^3, t^2)$ and let $\lambda: \mathbb{R}\to \mathbb{R}_+$ be the
constant function defined by $\lambda(t)=1$.
The circle family $\mathcal{C}_{(\gamma, \lambda)}$
and the candidates of its envelope is depicted in
Figure \ref{figure_example9}.
\begin{figure}[h]
\begin{center}
\includegraphics[width=10cm]
{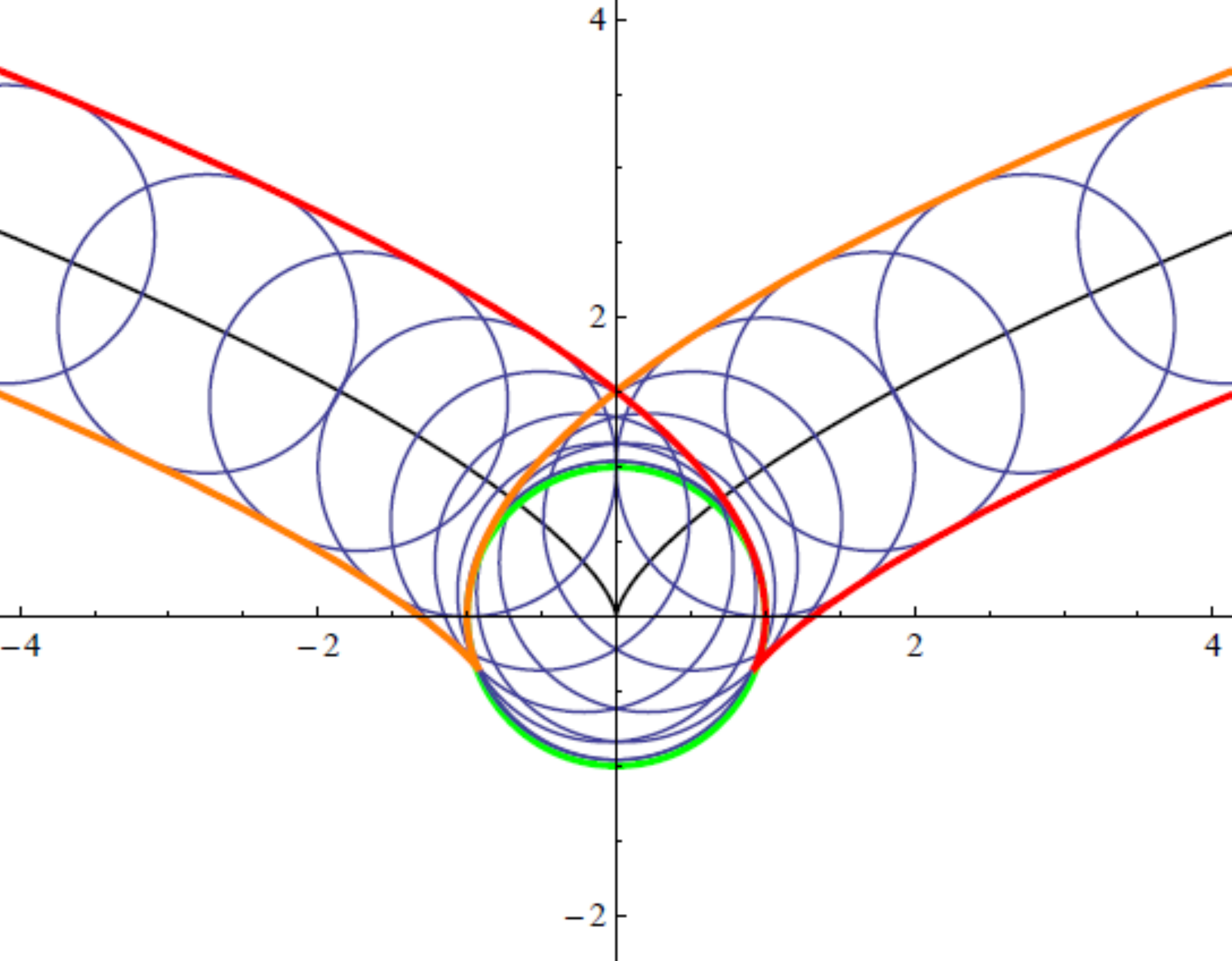}
\caption{The circle family $\mathcal{C}_{(\gamma, \lambda)}$
and the candidate of its envelope.
}
\label{figure_example9}
\end{center}
\end{figure}
It is easily seen that the mapping
$\nu: \mathbb{R}\to S^1$ defined by
$\nu(t)=\frac{1}{\sqrt{4+9t^2}}\left(2, -3t\right)$ gives the Gauss mapping
for $\gamma$.    Thus, $\gamma$ is a frontal.
By definition, the mapping $\mu: \mathbb{R}\to S^1$ has the form
$\mu(t)=\frac{1}{\sqrt{4+9t^2}}\left(3t, 2\right)$.
By calculation, we have
\[
\beta(t)=\frac{d \gamma}{dt}(t)\cdot \mu(t)=t\sqrt{4+9t^2}.
\]
Since the radius function $\lambda$ is constant, it follows
$\frac{d\lambda}{dt}(t)=0$.    Thus, for any $t\in \mathbb{R}$,
the unit vector $\widetilde{\nu}(t)$ satisfying
\[
\frac{d\lambda}{dt}(t) = -\beta(t)\left(\widetilde{\nu}(t)\cdot \mu(t)\right) 
\]
always exists.   Namely we have
\[
\widetilde{\nu}(t)=\pm\nu(t)=\frac{\pm 1}{\sqrt{4+9t^2}}\left(2, -3t\right).
\]
Thus, by the assertion (1) of Theorem \ref{theorem1},
$\mathcal{C}_{(\gamma, \lambda)}$ creates an envelope, and
the created envelope $f: \mathbb{R}\to \mathbb{R}^2$
has the following form
by the assertion (2) of Theorem \ref{theorem1}.
\[
f(t)=\gamma(t)+\lambda(t)\widetilde{\nu}(t) =
\left(t^3, t^2\right)\pm\frac{1}{\sqrt{4+9t^2}}\left(2, -3t\right)
= \left(t^3\pm \frac{2}{\sqrt{4+9t^2}}, t^2\mp \frac{3t}{\sqrt{4+9t^2}}\right).
\]
Finally, by the assertion (3-ii) of Theorem \ref{theorem1}, 
there are no other envelopes
created by $\mathcal{C}_{(\gamma, \lambda)}$.
\end{example}
\section{Alternative definitions}\label{section4}
In Definition \ref{definition_envelope} of Section \ref{section1},
the definition of envelope created by the circle family
is given.
In \cite{brucegiblin}, the set consisting of the images of envelopes defined
in Definition \ref{definition_envelope}
is called \textit{$E_2$ envelope} (denoted by $E_2$)
and two alternative definitions
(called \textit{$E_1$ envelope} and \textit{$\mathcal{D}$ envelope})
are given as follows.
\begin{definition}[$E_1$ envelope \cite{brucegiblin}]
{\rm
Let $\gamma: I\to \mathbb{R}^2$, $\lambda: I\to \mathbb{R}_+$
be a frontal and a positive function respectively.
Let $t_0$ be a parameter of $I$ and fix it.
Assume that
\[
\lim_{\varepsilon\to 0}
C_{(\gamma(t_0), \lambda(t_0))}\cap
C_{(\gamma(t_0+\varepsilon), \lambda(t_0+\varepsilon))}
\]
is not the empty set and denote the set by $I(t_0)$.
Take one point $e_1(t_0)=(x(t_0), y(t_0))$ of $I(t_0)$.
Then, the set consisting of the images of smooth mappings
$e_1: I\to \mathbb{R}^2$, if exists, is called
an \textit{$E_1$ envelope} created by the circle family
$\mathcal{C}_{(\gamma, \lambda)}$ and is denoted by $E_1$.
}
\end{definition}
\begin{definition}[$\mathcal{D}$ envelope \cite{brucegiblin}]
{\rm
Let $\gamma: I\to \mathbb{R}^2$, $\lambda: I\to \mathbb{R}_+$
be a frontal and a positive function respectively.
Set
\[
F(x, y, t)=||(x, y)-\gamma(t)||^2-\left(\lambda(t)\right)^2.
\]
Then, the following set is called
the \textit{$\mathcal{D}$ envelope} created by the circle family
$\mathcal{C}_{(\gamma, \lambda)}$ and is denoted by $\mathcal{D}$.
\[
\left\{(x, y)\in \mathbb{R}^2\, |\,
\exists t\in I \mbox{ such that } F(x, y, t)=\frac{\partial F}{\partial t}(x, y, t)=0
\right\}.
\]
}
\end{definition}
Concerning the relationships among $E_1$, $E_2$ and $\mathcal{D}$
for a given circle
family $\mathcal{C}_{(\gamma, \lambda)}$,
the following is known.
\begin{fact}[\cite{brucegiblin}]
$E_1\subset \mathcal{D}$ and $E_2\subset \mathcal{D}$.
\end{fact}
In this section, we study more precise relationships among
$E_1$, $E_2$ and $\mathcal{D}$.
\subsection{The relationship
between $E_1$ and $E_2$}\label{subsection4.1}
We first establish the relationship
between $E_1$ and $E_2$ as follows.
\begin{theorem}\label{theorem2}
$E_1=E_2$.
\end{theorem}
\begin{proof}
We first show $E_1\subset E_2$.
Let $t_0$ be a parameter of $I$ and let
$\left\{t_i\right\}_{i = 1, 2, \ldots}$ be a sequence of $I$ conversing
to $t_0$.
Take a point $(x(t_0), y(t_0))$ of $E_1$.
Then, we may assume that a point $(x(t_i), y(t_i))$
is taken from the intersection of two circles
$C(\gamma(t_i), \lambda(t_i))\cap C(\gamma(t_0), \lambda(t_0))$
and satisfies
\[
\lim_{t_i\to t_0}(x(t_i), y(t_i))=(x(t_0), y(t_0)).
\]
Then, we have the following.
\begin{eqnarray}
||(x(t_i), y(t_i))-\gamma(t_i)||^2 & = & \left(\lambda(t_i)\right)^2 \\
||(x(t_i), y(t_i))-\gamma(t_0)||^2 & = & \left(\lambda(t_0)\right)^2.
\end{eqnarray}
For $j=0, 1, 2, \ldots$,
set $\gamma(t_j)=\left(\gamma_x(t_j), \gamma_y(t_j)\right)$.
Subtracting (2) from (1) yields the following.
\begin{eqnarray*}
{ } & { } &
-2\left(
x(t_i)\left(\gamma_x(t_i)-\gamma_x(t_0)\right)
+
y(t_i)\left(\gamma_y(t_i)-\gamma_y(t_0)\right)
\right)
+\left(\gamma_x(t_i)\right)^2-\left(\gamma_x(t_0)\right)^2
+\left(\gamma_y(t_i)\right)^2-\left(\gamma_y(t_0)\right)^2 \\
{ } & = & \left(\lambda(t_i)\right)^2-\left(\lambda(t_0)\right)^2.
\end{eqnarray*}
Since $\lim_{i\to \infty}t_i=t_0$ and
$\lim_{t_i\to t_0}(x(t_i), y(t_i))=(x(t_0), y(t_0))$,
this equality implies
\[
-2\left(x(t_0)\frac{d\gamma_x}{dt}(t_0)+y(t_0)\frac{d\gamma_y}{dt}(t_0)\right)
+2\left(\gamma_x(t_0)\frac{d\gamma_x}{dt}(t_0)+
\gamma_y(t_0)\frac{d\gamma_y}{dt}(t_0)\right)
=2\lambda(t_0)\frac{d\lambda}{dt}(t_0).
\]
Hence we have
\[
-\frac{1}{\lambda(t_0)}\left({x(t_0)-\gamma_x(t_0)},
{y(t_0)-\gamma_y(t_0)} \right)
\cdot
\left(\frac{d\gamma_x}{dt}(t_0), \frac{d\gamma_y}{dt}(t_0)\right)
=
\frac{d\lambda}{dt}(t_0).
\]
Notice that the vector
$\frac{1}{\lambda(t_0)}\left({x(t_0)-\gamma_x(t_0)},
{y(t_0)-\gamma_y(t_0)} \right)=
\frac{1}{\lambda(t_0)}\left((x(t_0), y(t_0))-\gamma(t_0)\right)$
is a unit vector and
$\left(\frac{d\gamma_x}{dt}(t_0), \frac{d\gamma_y}{dt}(t_0)\right) =
\beta(t_0)\mu(t_0)$.   Thus
the creative condtion is satisfied at $t=t_0$.
Therefore, by the proof of the assertion (1) of Theorem \ref{theorem1}, 
the point
$\left(x(t_0), y(t_0)\right)$ must belong to $E_2$.
\par
\smallskip
Conversely,  
suppose that the circle family
$\mathcal{C}_{(\gamma, \lambda)}$ creates an $E_2$ envelope
$f: I\to \mathbb{R}^2$.
By the assertion (2) of Theorem \ref{theorem1}, 
$f$ has the following representation.
\[
f(t)=\gamma(t)+\lambda(t)\widetilde{\nu}(t).
\]
For a point $P\in \mathbb{R}^2$ and a unit vector ${\bf v}\in S^1$,
the straight line $L(P, {\bf v})$ is naturally defined as follows.
\[
{L}_{(P, {\bf v})}=\left\{
(x, y)\in \mathbb{R}^2\, |\, \left((x, y)-P\right)\cdot {\bf v}=0
\right\}.
\]
Then, since
\[
\frac{d f}{dt}(t)\cdot \widetilde{\nu}(t)=
\left(\frac{d\gamma}{dt}(t)+\frac{d\lambda}{dt}(t)\cdot\widetilde{\nu}(t)+
\lambda(t)\frac{d\widetilde{\nu}}{dt}(t)\right)\cdot\widetilde{\nu}(t)
=\frac{d\gamma}{dt}(t)\cdot \widetilde{\nu}(t)+\frac{d\lambda}{dt}(t)
=\beta(t)\left(\mu(t)\cdot \widetilde{\nu}(t)\right)+\frac{d\lambda}{dt}(t)=0,
\]
$f$ is an $E_2$ envelope created by the straight line family
\[
\mathcal{L}_{(f, \widetilde{\nu})}=\left\{
L_{(f(t), \widetilde{\nu}(t))}
\right\}_{t\in \mathbb{R}}.
\]
Take one parameter $t_0\in I$ and let $\{t_i\}_{i=1, 2, \ldots }\subset I$
be a sequence converging to $t_0$.
Since for the straight line family  $\mathcal{L}_{(f, \widetilde{\nu})}$
the image of 
$E_2$ envelope is the same as $E_1$ envelope
(see the assertion (c) of Theorem 1 in \cite{nishimura}),
for any sufficiently large $i\in \mathbb{N}$ there exists a point
\[
\left(x(t_i), y(t_i)\right)\in
L_{(f(t_0), \widetilde{\nu}(t_0))}\cap L_{(f(t_i), \widetilde{\nu}(t_i))}
\]
such that $\lim_{i\to \infty}\left(x(t_i), y(t_i)\right)=f(t_0)$.
Hence for any sufficiently large $i\in \mathbb{N}$ there must exist a point
\[
\left(\widetilde{x}(t_i), \widetilde{y}(t_i)\right)\in
C_{(\gamma(t_0), \lambda(t_0))}\cap C_{(\gamma(t_i), \lambda(t_i))}
\]
such that $\lim_{i\to \infty}\left(\widetilde{x}(t_i), \widetilde{y}(t_i)\right)=f(t_0)$
(see Figure \ref{figure_theorem2}).
\begin{figure}[h]
\begin{center}
\includegraphics[width=10cm]
{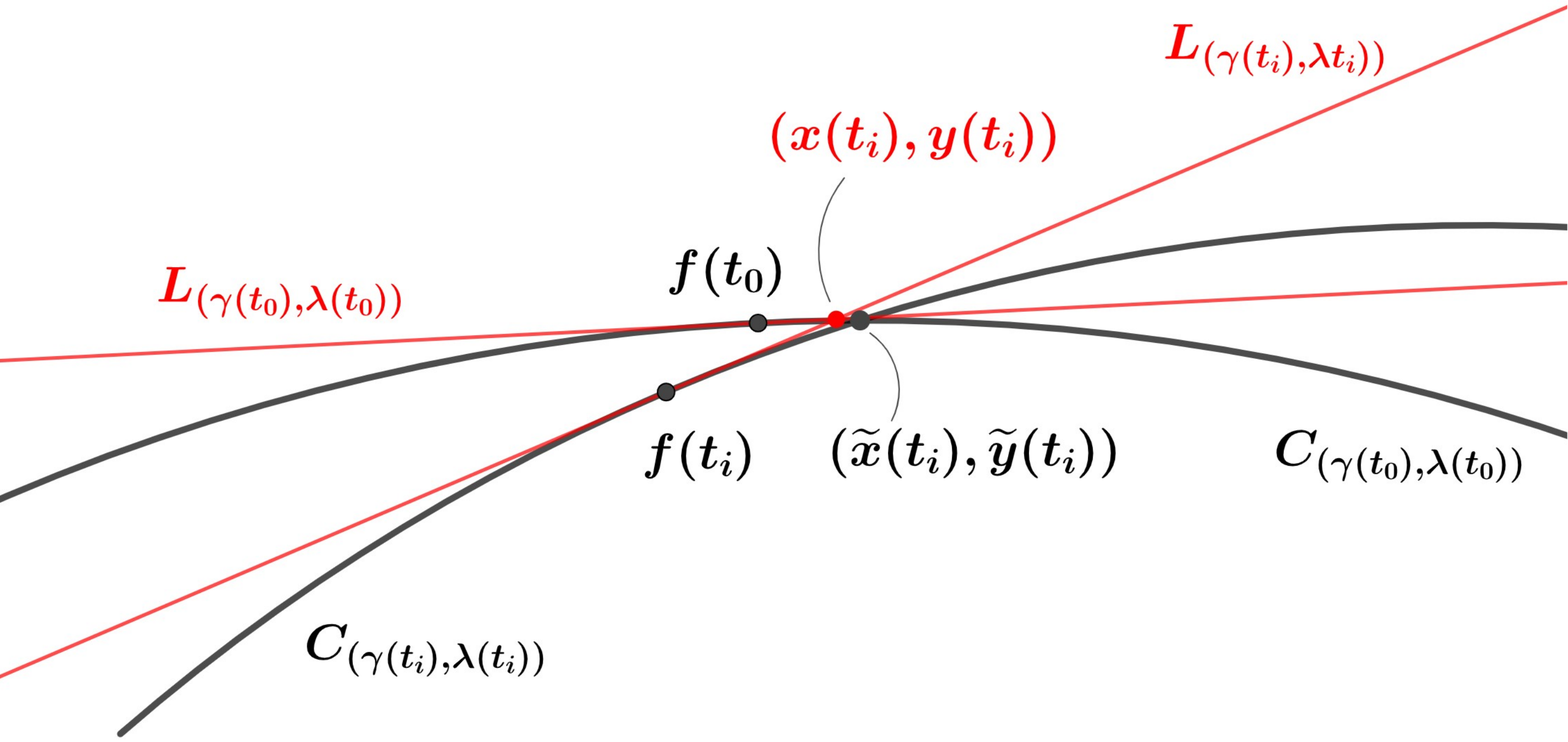}
\caption{Existence of
$\left(\widetilde{x}(t_i), \widetilde{y}(t_i)\right)\in
C_{(\gamma(t_0), \lambda(t_0))}\cap C_{(\gamma(t_i), \lambda(t_i))}
$ satisfying 
$\lim_{i\to \infty}\left(\widetilde{x}(t_i), \widetilde{y}(t_i)\right)=f(t_0)$.
}
\label{figure_theorem2}
\end{center}
\end{figure}
 Therefore, the point
$f(t_0)\in \mathbb{R}^2$ belongs to $E_1$.
Since $f$ is an arbitrary envelope created by
$\mathcal{C}_{(\gamma, \lambda)}$ and
$t_0$ is an arbitrary parameter in $I$,  it follows that $E_2\subset E_1$.
\end{proof}
\subsection{A relationship
between $E_2$ and $\mathcal{D}$}\label{subsection4.2}
In this subsection, we prove the following theorem which asserts that
$\mathcal{D}=E_2$ if and only if $\gamma: I\to \mathbb{R}^2$ is non-singular,
and $\mathcal{D}$ contains not only $E_2$ but also the circle
$C_{(\gamma(t), \lambda(t))}$ at a singular point $t$ of $\gamma$
when $\gamma$ is singular.
\begin{theorem}\label{theorem3}
Let $\gamma: I\to \mathbb{R}^2$, $\lambda: I\to \mathbb{R}_+$
be a frontal and a positive function respectively.
Suppose that the circle family
$\mathcal{C}_{(\gamma, \lambda)}$ is creative.
Then, the following holds.
\[
\mathcal{D}=E_2\cup
\left(\bigcup_{t\in \Sigma(\gamma)}C_{(\gamma(t), \lambda(t))}\right).
\]
Here, $\Sigma(\gamma)$ stands for the set consisting of
singular points of $\gamma: I\to \mathbb{R}^2$.
\end{theorem}
\begin{proof}
Recall that
\[
\mathcal{D}=
\left\{(x, y)\in \mathbb{R}^2\, |\,
\exists t\in I \mbox{ such that } F(x, y, t)=\frac{\partial F}{\partial t}(x, y, t)=0
\right\}.
\]
Let $\left(x_0, y_0\right)$ be a point of $\mathcal{D}$.
Since $F(x, y, t)=||(x, y)-\gamma(t)||^2 - |\lambda(t)|^2$,
it follows that there exists a $t_0\in I$ 
such that the following (a) and (b) are satisfied.
\begin{enumerate}
\item[(a)] 
$\left((x_0, y_0)-\gamma(t_0)\right)
\cdot \left((x_0, y_0)-\gamma(t_0)\right)-\left(\lambda(t_0)\right)^2=0$.
\item[(b)] $\frac{d\left(\left((x_0, y_0)-\gamma(t_0)\right)
\cdot \left((x_0, y_0)-\gamma(t_0)\right)-\left(\lambda(t_0)\right)^2
\right)}{dt}=0.$
\end{enumerate}
The condition (a) implies that there exists a 
unit vector $\nu_1(t_0)\in S^1$ 
such that
\[
\left(x_0, y_0\right)=\gamma(t_0)-\lambda(t_0)\nu_1(t_0).
\]
The condition (b) implies 
\[
\frac{d\gamma}{dt}(t_0)\cdot \left(\left(x_0, y_0\right)-\gamma(t_0)\right)
-
\frac{d\lambda}{dt}(t)\lambda(t_0)=0.
\]
Since $\frac{d\gamma}{dt}(t_0)=\beta(t_0)\mu(t_0)$, just by substituting,
we have the following. 
\[
\lambda(t_0)\left(
\beta(t_0)\left(\mu(t_0)\cdot\nu_1(t_0)\right)+\frac{d\lambda}{dt}(t_0)
\right)=0.
\]
Since $\lambda(t)>0$ for any $t\in I$, it follows
\[
\frac{d\lambda}{dt}(t_0)=-\beta(t_0)\left(\mu(t_0)\cdot\nu_1(t_0)\right).
\]
\par
On the other hand, since $\mathcal{C}_{(\gamma, \lambda)}$ is creative,
there must exist a smooth unit vector field
$\widetilde{\nu}: I\to S^1$ along $\gamma: I\to \mathbb{R}^2$ such that
\[
\frac{d\lambda}{dt}(t)=-\beta(t)\left(\mu(t)\cdot\widetilde{\nu}(t)\right)
\]
for any $t\in I$.
Suppose that the parameter $t_0\in I$ is a regular point of $\gamma$.
Then, $\beta(t_0)\ne 0$.   
Therefore, by the proof of the assertion (1) of Theorem \ref{theorem1}, 
it follows
\[
\left(x_0, y_0\right)\in E_2.
\]
Suppose that the parameter $t_0\in I$ is a singular point of $\gamma$.
Then, $\beta(t_0)=0$. 
Thus, for any unit vector
${\bf v}\in S^1$, the following holds. 
\[
\frac{d\lambda}{dt}(t_0)=-\beta(t_0)\left(\mu(t_0)\cdot{\bf v}\right).
\]
Hence, at the singular point $t_0\in I$ of $\gamma$, we may choose
any unit vector ${\bf v}\in S^1$ as the unit vector $\nu_1(t_0)$.
Therefore, 
it follows 
\[
\mathcal{D}_0= C_{(\gamma(t), \lambda(t))}, 
\]
where $
\mathcal{D}_0=
\left\{(x, y)\in \mathbb{R}^2\, |\,
F(x, y, t_0)=\frac{\partial F}{\partial t}(x, y, t_0)=0
\right\}.
$
\end{proof}
\section*{Acknowledgement}
{
~~~This work was supported 
by the Research Institute for Mathematical Sciences, 
a Joint Usage/Research Center located in Kyoto University. 
\par
The first author is supported by the National Natural Science Foundation of
China (Grant No. 12001079), 
Fundamental Research Funds for the Central Universities 
(Grant No. 3132023205) and China Scholarship Council.     
The second author is 
supported by JSPS KAKENHI (Grant No. 23K03109).   
}
%
%

\end{document}